\documentclass[11pt]{article}
\usepackage[final]{epsfig}
\usepackage{graphics}
\usepackage{amsmath}
\usepackage{amsfonts}
\usepackage{latexsym}
\usepackage{amssymb}
\usepackage{graphicx}
\usepackage{ stmaryrd }
\usepackage{dsfont}
\usepackage{amsthm}

\newtheorem{lemma}{Lemma}

\newtheorem*{theorem}{Theorem}
\newtheorem{definition}{Definition}

\newtheorem*{conjecture}{Conjecture}

\begin{document}

\title{On equidissection of balanced polygons}
\author{Daniil Rudenko}
\maketitle

\begin{abstract}

In this paper we show that a lattice balanced polygon of  odd area cannot be cut into an odd number of triangles of equal areas. First result of this type was obtained by Paul Monsky  in 1970.  He proved that a square cannot be cut into  an odd number of triangles of equal areas. In 2000 Sherman Stein conjectured that the same holds for any balanced polygon. 

We also show  connections between the equidissection problem and  tropical geometry.
\end{abstract}

\section{Equidissection problem}
\begin{theorem} [P. Monsky, 1970]
A square cannot be cut\footnote{By the phrase  \textit{polygon B is cut into triangles} we mean that B can be presented as a union of a finite number of triangles  so that the interiors of the triangles have an empty intersection with each other.  Fig. \ref{triangulation} illustrates this.} into an odd number of triangles\footnote{Throughout this article "triangle" is taken to include the degenerate case.} of equal areas.
\end{theorem}

\begin{figure}[hbtp]
\centering
\includegraphics[width=2.7in, height=2.4in]{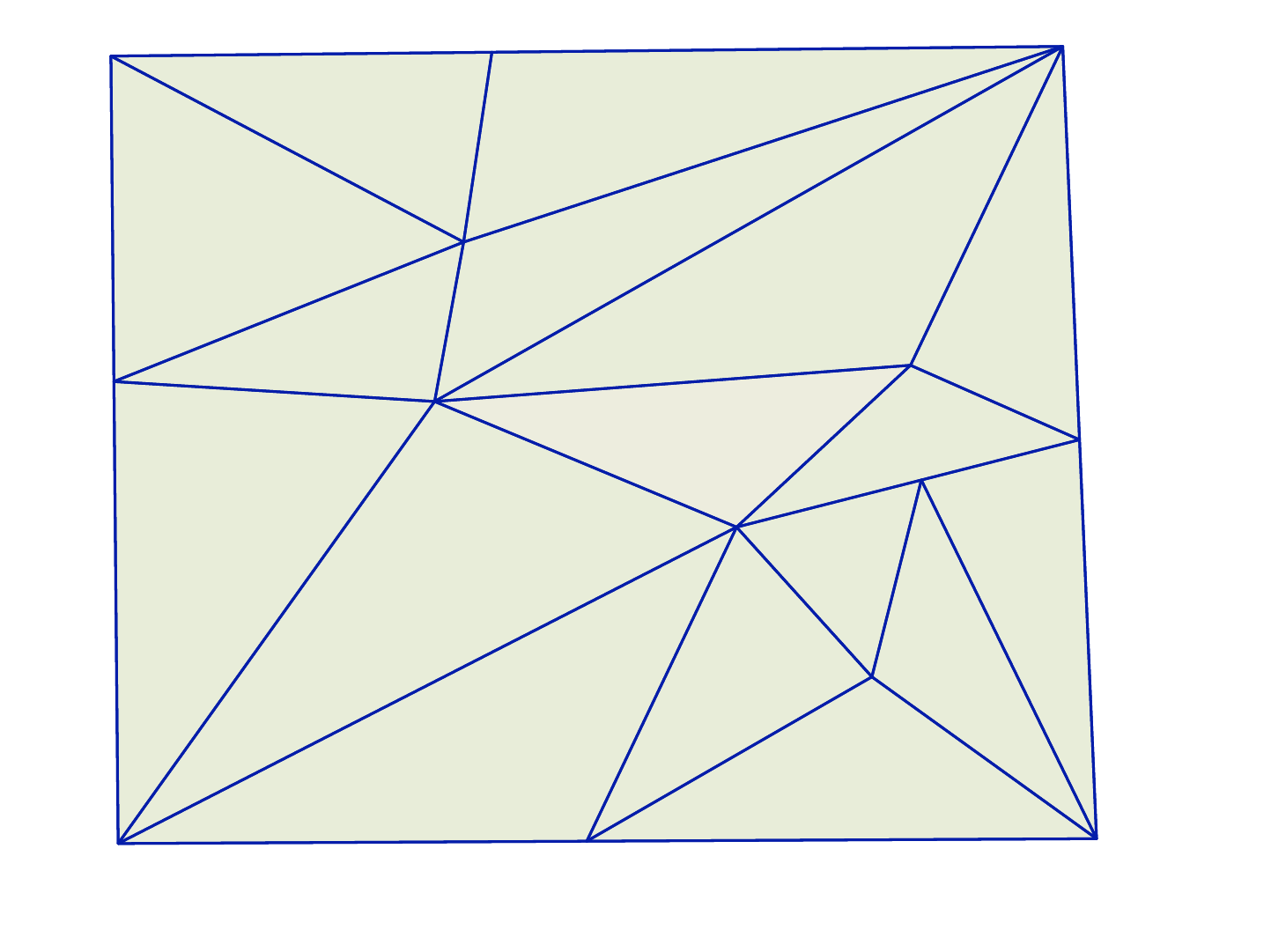}
\caption{A square is cut into triangles}
\label{triangulation}
\end{figure}	

The only known proof of this theorem was published by Monsky in 1970 \cite{Mo1}. The proof is based on two key ideas: the Sperner's Lemma and the coloring of the plane in three colors based on a 2--adic valuation.

After that, several generalizations of Monsky's results appear.  The first generalization was conjectured by Stein and proved by Monsky in 1990 \cite{Mo2}. It claims that a centrally symmetric polygon cannot be cut into an odd number of triangles of equal areas.  Though it is based on the same idea of 3-coloring, this proof is technically  more challenging than the proof in the case of the square and uses a non-trivial homological technique.  

In 1994 Bekker and Netsvetaev proved similar statement in higher dimensions \cite{BM}.

 To state another generalization we need a definition. Let us call  a finite union of squares of area 1 with integer coordinates of vertices a \textit{polyomino}.  First, Stein proved in 1999      \cite{St3} that a polyomino of  odd area cannot be cut into an odd number of triangles of equal areas, and in 2002 Praton \cite{Pr1} proved the same for an even-area polyomino.
In 2000  Stein \cite{St1} made  a conjectural generalization of Theorem 1, see also \cite{St2}.

Let $B$ be a plane polygon with clockwise  oriented boundary.  $B$ is called  \textit{balanced} if its edges can be divided into pairs so that in each pair edges are parallel, equal in length and have opposite orientation (the edges are oriented, their orientation comes from the orientation of the boundary).

Now we are ready to formulate the Stein Conjecture.

\begin{conjecture} [S. Stein, 2000] \label{stein conjecture}
A balanced polygon cannot be cut into an odd number of triangles of equal areas.
\end{conjecture}

In this note we will present a proof of a partial case of Conjecture \ref{stein conjecture}. Namely, we will prove the following theorem.

\begin{theorem} [Non-equidissectibility of a balanced lattice polygon] \label{main}  
Consider a balanced  polygon $B$ of the integer odd area and assume that the coordinates of all the vertices are integer numbers. Then $B$  cannot be cut into an odd number of triangles of equal areas.
\end{theorem}	
For an example of a balanced lattice polygone of area 15, see Fig. \ref{lattice}.

\begin{figure}[hbtp]
\centering
\includegraphics[width=3in]{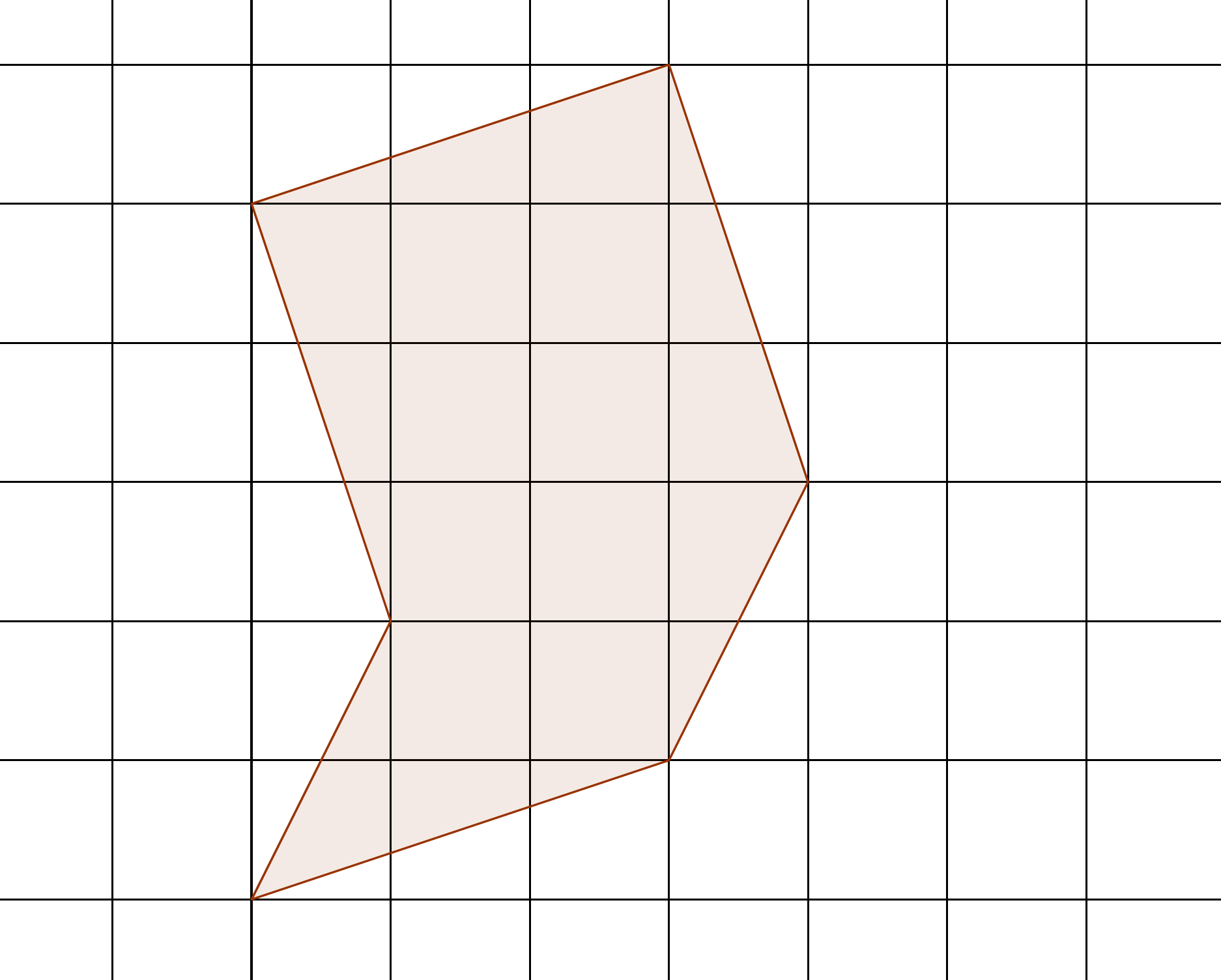}
\caption{Balanced lattice polygon of area 15}
\label{lattice}
\end{figure}

The proof of  nonexistence of equidissection of a balanced lattice polygon consists of several steps. 

In section $2$ we review the coloring of the plane in three colors introduced by Monsky.  

In section $3$ we introduce  a notion of a \textit{degree} of a broken line. It is an  integer  that depends both on a coloring and a broken line. We prove that if a polygon can be cut into triangles with nonnegative $2$--adic valuations of araes, then its degree  is $0$.  

In section $4$ previous results are applied to the case of a lattice polygon.

The proof of the nonexistence of equidissection of a balanced lattice polygon is finished in section $5$.

In the  appendix we show connections between tropical geometry and $3$-colorings of projective plane.

\parskip 9pt
{\bf Acknowledgments}.
My gratitude goes to Sergei Tabachnikov for inspiring me to write this article. Also to Sherman Stein for proposing the conjecture  and for his constructive criticism of my nascent ideas. 
I am especially grateful to Nikolai Mnev, without whose guidance and support this article would not have been possible.

        \section{Tropical colorings}
The main tool for us will be a special type of coloring of a plane in 3 colors. To begin with, let us recall the notion of a discrete valuation and sketch its basic properties.
A function $\nu_2  : \mathds{R}  \longrightarrow  \mathds{R} \cup \{ \infty \} $ is a 2--adic valuation on the field of real numbers if for any two numbers $a,b \in \mathds{R}$ the following properties hold:
\begin{description}
\item{Property 1:} $\nu_2(ab)=\nu_2(a) + \nu_2(b)$, $\nu_2(\frac{a}{b})=\nu_2(a) - \nu_2(b)$,
\item{Property 2:} $\nu_2(a+b)\geqslant \it{min}\{\nu_2(a) , \nu_2(b)\}$,
\item{Property 3:} if $\nu_2(a)<\nu_2(b)$ then $\nu_2 (a+b)=\nu_2(b)$,
\item{Property 4:} $ \nu_2 (0) = \infty$,
\item{Property 5:} It extends standart 2--adic valuation on rationals:
$$  \forall q \in \mathds{Q}  \setminus \{0\} \;  \; \nu_2(q)=s \iff q=2^s\frac{2k+1}{2l+1} \text{ for some } k, l, s \in \mathds{Z}.  $$
\end{description}

The existence of such a function follows from the theorem of the extension of valuations,  see  \cite{La1}. This function is not unique and its construction is based on the Axiom of Choice.

Our goal now is to construct a family of 3--colorings of a plane (we will call these colorings "tropical")  with two properties:

\textbf{(P 1)} On any line points of only two colors occur.

\textbf{(P 2)} For any triangle  with vertices having all 3 different colors its area has a negative 2--adic valuation.

\begin{figure}[hbtp]
\centering
\includegraphics[width=3in]{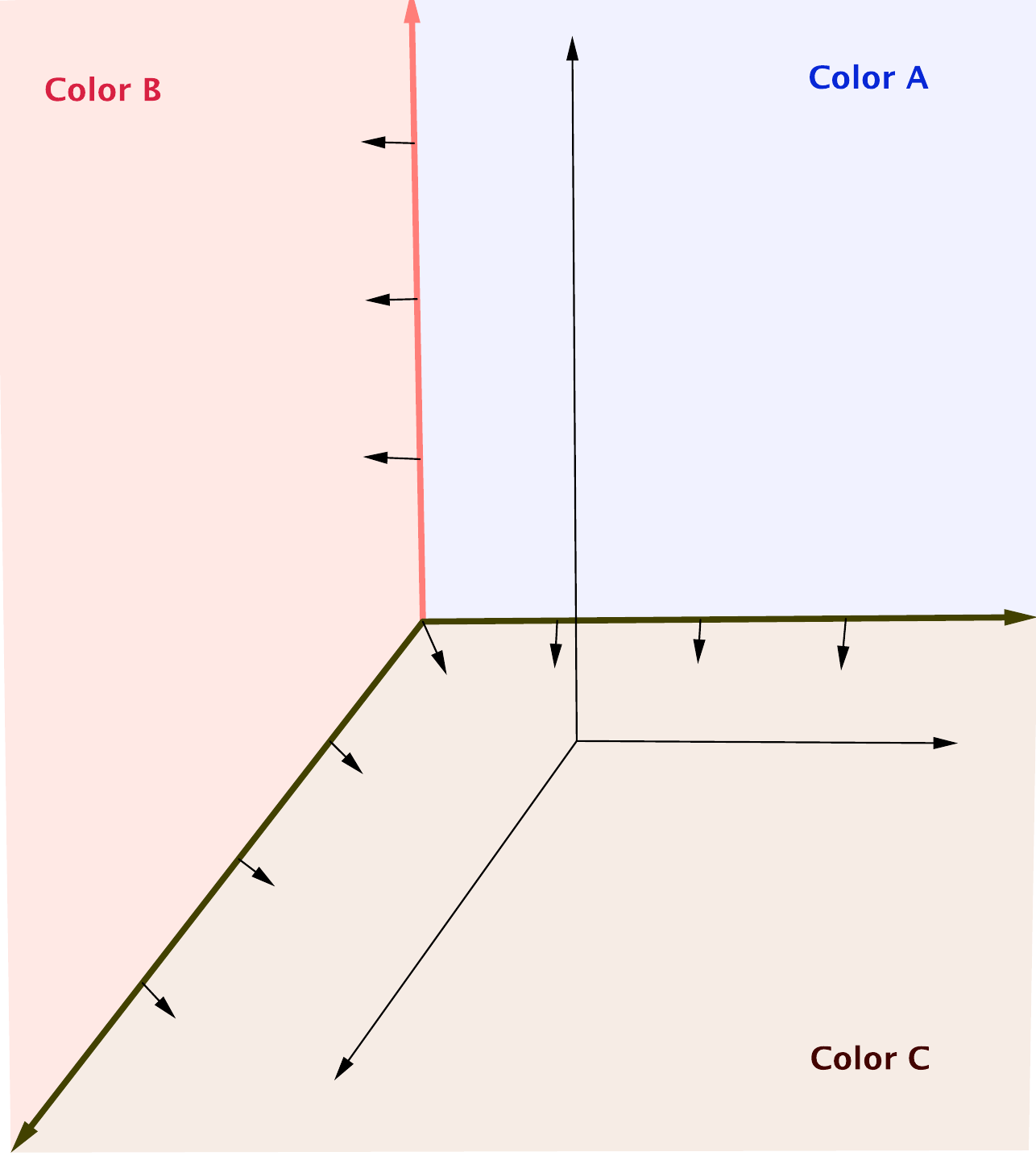}
\caption{Tropical coloring of the plane and the image of the line}
\label{Tropical coloring}
\end{figure}

Let us color  points in the plane in three colors $A, B, C$ according to the following rule: a  point $Z$ with coordinates $(x,y)$ is colored

 in color $A$, if $ \nu_2(x)>0, \nu_2(y)>0$, 

 in  color $B$, if $\nu_2(y)\leqslant 0, \nu_2(x)>\nu_2(y)$, 
 
 in  color $C$, if $ \nu_2(x)\leqslant 0, \nu_2(y)\geqslant \nu_2(x)$.\\
 
 This defines a map from the plane to a three-element set $$\pi: \mathds{R}^2  \longrightarrow\{A,B,C\}.$$ 
On Fig. 3 the way of coloring is presented in coordinates $\nu_2(x), \nu_2(y)$.

 For any area-preserving affine transformation $ \mathcal{A}\in \mathds{R}^2 \rtimes SL_2(\mathds{R})$, we can define another coloring $\pi^\mathcal{A}$ by the rule
 $$\pi^\mathcal{A}(Z)=\pi(\mathcal{A}(Z)), \text{for any point Z}.$$ 
This defines a family of 3--colorings, which we will call tropical.

\begin{lemma} \label{one}
For the  3--coloring $\pi^\mathcal{A}$  properties \textbf{P1} and \textbf{P2} hold.
 \end{lemma}
 
 \begin{proof}

\textbf{P2} $\implies$ \textbf{P1}. \\ If there were three points of different colors on the same line, they would form a  triangle of area $0$.  Because $\nu_2(0)=\infty$, this is impossible.
\\
\textbf{P2}.

For a coloring $\pi^\mathcal{A}$ we need to prove that for any triangle $\triangle$, whose image under $\mathcal{A}$ has vertices of all three different colors, the following holds true:
$$\nu_2(Area(\triangle))<0.$$
Since $\mathcal{A}$ is area-preserving, it is enough to prove that $\mathcal{A}(\triangle)$ has  area with negative valuation. 
Suppose that the triangle $\mathcal{A}(\triangle) $ has vertex $(x_1,y_1)$ of color A, $(x_2,y_2)$ of color B and $(x_3,y_3)$ of color C. Then, its area is equal to 
$$Area(\triangle)=\frac{1}{2}\begin{vmatrix}
  x_1 & y_1 & 1 \\
  x_2 & y_2 & 1 \\
  x_3 & y_3 & 1 
 \end{vmatrix} = \frac{1}{2} \Big( (x_2-x_1)(y_3-y_1)-(y_2-y_1) (x_3-x_1)  \Big).$$
Application of the properties of valuation and the definition of the coloring leads to 
$$\nu_2\Big(  \frac{1}{2} (y_2-y_1)(x_3-x_1)   \Big)=-1+\nu_2(y_2-y_1)+\nu_2(x_3-x_1) =$$
$$=-1+\nu_2(y_2)+\nu_2(x_3)<-1+min\{\nu_2(x_2),0\}+min\{\nu_2(y_3),0\}\leq$$
$$\leq -1+ \nu_2(x_2-x_1)+\nu_2(y_3-y_1)=\nu_2\Big(  \frac{1}{2} (x_2-x_1)(y_3-y_1)   \Big).$$
Therefore, 
$$\nu_2(\mathcal{A}(\triangle))=\nu_2\Big(  \frac{1}{2} (y_2-y_1)(x_3-x_1)   \Big)=-1+\nu_2(y_2)+\nu_2(x_3)\leq-1+0+0=-1.$$
 \end{proof}

 \section{Degree of a broken line}
Given a tropical coloring $\pi^\mathcal{A}$, one can construct a degree map associated with it. It assigns an integer number to any oriented broken line.

  Let $K_n$ be a complete graph with $n$ vertices  considered as 1--dimensional simplicial complex. Suppose that we are given a 1--dimensional simplicial complex $K$ and a map $Col$ sending vertices of $K$ to vertices of $K_n$. Then, this map can be extended to a continuous map from complex $K$ to $K_n$ according to the following rules: 
 \begin{itemize}
\item Vertex $X$ is sent to   point $ Col(X)$.
\item Edge $X Y$ is sent to edge $Col(X)Col(Y)$ by a linear map determined by  its ends.
\end{itemize}
This map is, obviously, a continuous simplicial map from one simplicial complex to the other. We will use the same letter for both the original map (coloring) and the extended one.

For the following, let us fix a tropical coloring $\pi^\mathcal{A}$.
    
 \begin{figure}[hbtp]
\centering
\includegraphics[width=5in]{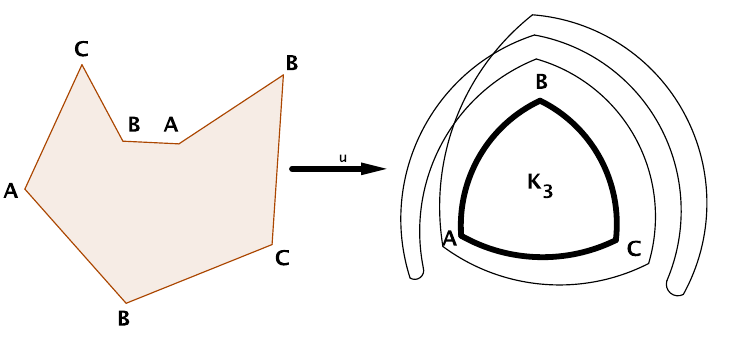}
\caption{Degree of a polygon is $-1$}
\label{class}
\end{figure} 

\begin{definition}
Any closed broken line  $ L= L_1 L_2 ... L_n$    has a natural structure of a simplicial complex. This complex is homeomorphic to a circle. Its vertices are 3--colored by $\pi^\mathcal{A}$. The extension of $\pi^\mathcal{A}$  gives a continuous map  from topological circle $L$ to topological circle $K_3$. We denote by $Deg(L,\pi^\mathcal{A})$ its topological degree.  See Fig. \ref{class}.
 \end{definition}

For a polygon $M$ we denote  its boundary by $\partial M$.
 
 \begin{lemma} \label{two}
If a polygon  $M$ can be cut into triangles, whose areas have non--negative 2--adic valuations, then for every tropical coloring $\pi^\mathcal{A}$  
$$Deg( \partial M,\pi^\mathcal{A})=0.$$
 \end{lemma}

 \begin{proof}
 
Suppose that polygon $M$ has a triangulation $\mathcal{T}$ with each triangle having non-negative valuation of area. This triangulation carries a natural structure of a 1--dimensional simplicial complex, induced from the plane, with vertices colored by $\pi^\mathcal{A}$. $\partial M$ is a subcomplex of $\mathcal{T}$ homeomorphic to a circle, so there is a class  $[\partial M] \in H_1(\mathcal{T},\mathbb{Z})$ corresponding to $\partial M$. The degree $Deg(\partial M,\pi^\mathcal{A})$ is equal to the image of the class $[\partial M]$ under the map 
$$H_1(\mathcal{T},\mathbb{Z}) \xrightarrow{\pi^\mathcal{A}}H_1(K_3,\mathbb{Z})\cong \mathbb{Z}. $$
  One can orient all triangles in the cut in a coherent way. Then the triangles sharing an edge will induce opposite orientations on this edge. Adding up the classes of all triangles' boundaries we obtain the  class of $\partial M$
$$[\partial M]=\sum_{\triangle \in  \mathcal{T}}[\partial \triangle]$$
and applying $\pi^\mathcal{A}$ we get:
$$\pi^\mathcal{A}([\partial M])=\sum_{\triangle \in  \mathcal{T}}\pi^\mathcal{A}([\partial \triangle]).$$
Since any $  \triangle \in  \mathcal{T}$  has non--negative valuation of area, at least two of its  vertices are of the same color according to Lemma \ref{one}. So, $\pi^\mathcal{A}([\partial \triangle])=0$ for any triangle in $\mathcal{T}$ and thus

$$\pi^\mathcal{A}([\partial M])=Deg(\partial M,\pi^\mathcal{A})=0. $$ 
 \end{proof}

 \section{Lattice Polygons} 
Points with integer coordinates in a plane form a two dimensional lattice in $\mathds{R}^2$, we will call it $\mathcal{L}$.  We call a polygon $M$  or a closed broken line $L$ \textit{ lattice} if all its vertices have integer coordinates. 
  
Let us denote by $K_4$ a simplicial complex with four vertices labeled by  elements of the group $\mathds{Z}_2 \times \mathds{Z}_2$ and edjes connecting each two of its vertices.
We can map $\mathcal{L}$ to $K_4$ by the map $$\overline{*}:(x_1,x_2) \longrightarrow \overline{  (x_1,x_2) }=(x_1 \, mod \, 2, x_2 \, mod \, 2).$$ 
For any lattice broken line $L$ we can consider its image under the map $\overline{*}$ using the construction from the previous section. This map induces a map on simplicial homology groups:
$$\mathbb{Z} \cong H_1( L,\mathbb{Z}) \xrightarrow{\overline{*}}H_1(K_4,\mathbb{Z}) \cong \mathbb{Z} \oplus \mathbb{Z} \oplus \mathbb{Z}.$$
The image of $1 \in \mathbb{Z}$ in the group $H_1(K_4,\mathbb{Z}) $ will be denoted by $\langle  L \rangle$ and called the \textit{class} of broken line  $ L$.

 \begin{lemma} \label{three}
If a lattice polygon M can be dissected into triangles, whose areas have non-negative 2--adic valuations, then the map $\overline{*}$ sends 
$H_1(\partial M,\mathbb{Z})$ to $0$.
 \end{lemma}

 \begin{proof}
Let us denote the points of $K_4$ in the following way:
$$X_1=(0,0); \;X_2=(0,1); X_3=(1,0);\;X_4=(1,1);  $$
The three cycles 
$$\sigma_1=X_1X_2+X_2X_3+X_3X_1,$$
$$\sigma_2=X_1X_3+X_3X_4+X_4X_1,$$
$$\sigma_3=X_3X_2+X_4X_3+X_2X_4$$
generate $H_1(K_4,\mathbb{Z})\cong \mathbb{Z} \oplus \mathbb{Z} \oplus \mathbb{Z}.$
Let us suppose that $\langle \partial M \rangle=\lambda_1 \sigma_1+\lambda_2 \sigma_2+\lambda_3 \sigma_3 \in H_1(K_4,\mathbb{Z}).$

For an integer number its 2--adic valuation is always nonnegative and it is equal to zero if and only if the number is odd. From the definition of the 2--adic valuation it is clear that for a point $(x,y) \in \mathcal{L}$ 
$$\overline{  (x,y) }=(0,0) \in \mathds{Z}_2 \times \mathds{Z}_2 \implies \pi\Big( (x,y) \Big)= A,$$
$$\overline{  (x,y) }=(0,1) \in \mathds{Z}_2 \times \mathds{Z}_2 \implies\pi\Big( (x,y) \Big)= B, $$
$$\overline{  (x,y) }=(1,0) \in \mathds{Z}_2 \times \mathds{Z}_2 \implies \pi\Big( (x,y) \Big)= C, $$
$$\overline{  (x,y) }=(1,1) \in \mathds{Z}_2 \times \mathds{Z}_2 \implies \pi\Big( (x,y) \Big)= C. $$

This means that the map $\pi$ is correctly defined on $K_4 \cong \mathds{Z}_2 \times \mathds{Z}_2$ and $\pi\Big( (x,y) \Big)=\pi\Big( \overline{  (x,y) } \Big)$. Any area-preserving affine transformation $\mathcal{A}\in \mathds{Z}^2 \rtimes SL_2(\mathds{Z}) \subseteq \mathds{R}^2 \rtimes SL_2(\mathds{R}) $ acts on the four vertices of $K_4 \cong \mathds{Z}_2 \times \mathds{Z}_2$ by permutation, a simple check shows that 
$\pi^\mathcal{A}\Big( (x,y) \Big)=\pi\Big( \mathcal{A}\overline{  (x,y) } \Big)$.

We will apply Lemma \ref{two} to the three colorings corresponding to the following area--preserving affine transformations:

$E:   \; (x,y) \longrightarrow (x,y),$

$U:   \; (x,y) \longrightarrow (x+y,y),$

$V:   \; (x,y) \longrightarrow (y+1,x).$

By Lemma \ref{two},  $Deg(\partial M,\pi^E)=0$.
$$0=\pi^E(\langle \partial M \rangle)=\lambda_1 \pi^E(\sigma_1)+\lambda_2 \pi^E(\sigma_2)+\lambda_3 \pi^E(\sigma_3)=$$
$$=\lambda_1 (AB+BC+CA)+\lambda_2 (AC+CC+CA)+\lambda_3 \pi^E(CB+CC+BC)=$$
$$=\lambda_1 (AB+BC+CA).$$
So, $\lambda_1=0$. Similarly, applying the same procedure to affine transformations $U$ and $V$, we get $\lambda_2=0$ and $\lambda_3=0$. 

 \end{proof}

 \section{Balanced Polygons} 
For two vectors $v=(v_x,v_y)$ and $w=(w_x,w_y)$ their \textit{wedge product} is defined as the oriented area of the parallelogram formed by these vectors. It can be calculated as the determinant:
$$v \wedge w=\begin{vmatrix}
  v_x & w_x  \\
  v_y & w_y   
 \end{vmatrix} =v_x w_y-v_y w_x.$$

\begin{definition}
For any closed broken line  $L= L_1 L_2 ... L_n$  we define its generalised area by
$$Area(L)=\frac{1}{2}\sum_{i=1}^{n}\overline{OL_i}\wedge \overline{OL_{i+1}} , \text{ where  } L_{n+1}:=L_1.$$
 \end{definition}

For a non--selfintersecting broken line the notion defined above gives  the oriented area of a polygon, which is bounded by the broken line.

\begin{lemma} \label{three+}
For a lattice parallelogram $P$ the following is true:
\begin{itemize}
\item If area of $P$ is even, then $\langle \partial P \rangle=0.$
\item If area of $P$ is odd, then $\langle \partial P \rangle \in \{\pm(\sigma_2+\sigma_3), \pm(\sigma_3+\sigma_1), \pm(\sigma_1+\sigma_2 )\}.$
\end{itemize}
 \end{lemma}

 \begin{proof}
The parallelogram can be cut into two equal triangles of the integer area.  The application of Lemma \ref{two} gives the first statement.
 
To prove the second statement, we will show that if the parallelogram $P$ has  odd area, then all its pairs of coordinates of vertices are different modulo $2$.   If the vertices of the parallelogram have coordinates $(x_1,y_1), (x_2,y_2)$, $(x_3,y_3), (x_4,y_4)$, then its area is equal to
$$Area(P)= (x_2-x_1)(y_3-y_1)-(y_2-y_1) (x_3-x_1) =$$
$$= (x_4-x_1)(y_3-y_1)-(y_4-y_1) (x_3-x_1).$$
From this formula, it is clear that if there are two vertices with both $x$ and $y$ coordinates being conjugate modulo $2$, then $Area(P)$ is even.

So, vertices of the parallelogram are colored in colors A, B, C, C. Depending on the order in which these colors follow each other we obtain one of the cycles $\pm (\sigma_2+\sigma_3), $ $\pm(\sigma_3+\sigma_1), $ $\pm(\sigma_1+\sigma_2)$.
 \end{proof}

The following Lemma generalizes Lemma \ref{three+}

\begin{lemma} \label{four}
If $B$ is a balanced lattice polygon, then the image of its boundary under the map $\overline{*}$ is representing a class $\langle \partial B \rangle$ in the group
$H_1(K_4,\mathbb{Z}) \cong \mathbb{Z} \oplus \mathbb{Z} \oplus \mathbb{Z}$, which is lying in a subgroup of index $2$ generated by 
$\sigma_2+\sigma_3, \sigma_3+\sigma_1 \text{ and } \sigma_1+\sigma_2$:
$$\langle \partial B \rangle=\mu_1(\sigma_2+\sigma_3)+\mu_2(\sigma_3+\sigma_1)+\mu_3(\sigma_1+\sigma_2)$$
for some $\mu_1,\mu_2, \mu_3 \in \mathbb{Z}$.\\Furthermore, $$Area(B) \equiv \mu_1+\mu_2+\mu_3 \text{ (mod 2)}.$$
 \end{lemma}

 \begin{proof}

Parallelograms are basic examples of balanced polygons and we have seen that Lemma \ref{four} holds true for them. Now, we are going to show that any balanced polygon is built from parallelograms in some sense. For this we need to describe an action of  group $S_{n}$ on the set of broken lines.

For a  broken line $L= L_1 L_2 ... L_n$ ,  let us  denote by $v_i=\overline{L_i L_{i+1}} $ the side vector of $L$ (here  $L_{n+1}:=L_1$). Any $\sigma \in S_{n}$ acts on the set of broken lines  according to the rule:  $\sigma(L)=M, \text{ where M is a broken line }  M_1 M_2 ... M_n  $ with $M_1=L_1$, and for each $i$
 $$   M_{i+1}= L_1+v_{\sigma^{-1}(1)}+v_{\sigma^{-1}(2)}+...+ v_{\sigma^{-1}(i)}.$$
This action sends balanced broken lines to balanced and lattice to lattice.

Let $\tau_i$ denote a transposition $(i ,i+1)$. It is well known that a set $\{ \tau_i \mid 1\leq i \leq n-1\}$ generates $S_{n}$.
One can check that
$$Area(\tau_j (L))=Area (L) - v_{i+1} \wedge v_i $$
and 
$$\langle \tau_j (L) \rangle=\langle L \rangle -\langle P \rangle.$$
Here P is a parallelogram $L_i L_{i+1} L_{i+2} X$ , with $X= L_{i}+v_{i+1}.$ These properties guarantee that broken lines $L$ and $\tau_i(L)$ satisfy the conditions of Lemma \ref{four}  simultaneously.

Since lattice polygon $B$ is balanced, number $n$ of its vertices is even and sides of $B$ can be indexed by numbers $\alpha_1, \alpha_2,... ,\alpha_k$ and $\beta_1, \beta_2,... ,\beta_k$ so that $n=2k$ and sides with indices $\alpha_i$ and $\beta_i$ are parallel, equal in length and inherit opposite orientations from the polygons' boundary. Numbers $\alpha_i$ and $\beta_i$ are just natural  numbers from $1$ to $n=2k$, so one can consider permutation
$$\sigma=
\begin{pmatrix}
  \alpha_1 & \beta_1 & \alpha_2 & \beta_2 &...&\alpha_k & \beta_k \\
  1 & 2 & 3 & 4 & ...& 2k-1 & 2k
 \end{pmatrix}.$$

In the broken line $\sigma(\partial B)$ after any side with an odd number goes the side parallel and equal to it and having the opposite direction. Both the area and the class of $\sigma(\partial B)$ in $H_1(K_4,\mathbb{Z})$  is equal to $0$. Since $\sigma$ can be presented as a product of transpositions $\tau_i$,  $\partial B$ satisfies the conditions of Lemma \ref{four}. 
 \end{proof}
Now we are ready to finish the proof of our main result. 

\begin{proof}[Proof of non-equidissectibility of a balanced lattice polygon]
Suppose that for a balanced lattice polygon $B$ of integer odd area there exist a cut into an odd number of triangles of equal areas. If $Area(B)=S$ and the number of triangles is equal to $N$, then the area of each of them is $\frac{S}{N}$. Since $S$ and $N$ are odd numbers,
$$\nu_2(\frac{S}{N})=\nu_2(S)-\nu_2(N)=0.$$
According to Lemma \ref{three}, the class of the broken line $\partial B$ in $H_1(K_4,\mathbb{Z})\cong \mathbb{Z} \oplus \mathbb{Z} \oplus \mathbb{Z}$ is equal to $0$, and according to Lemma \ref{four} there exists $\mu_1,\mu_2, \mu_3 \in \mathbb{Z}$ for which 
$$\langle \partial B \rangle=\mu_1(\sigma_2+\sigma_3)+\mu_2(\sigma_3+\sigma_1)+\mu_3(\sigma_1+\sigma_2)=0.$$
Therefore, $\mu_1=\mu_2=\mu_3=0$ and $Area(B)=S\equiv \mu_1+\mu_2+\mu_3 \equiv 0 \text{ (mod 2)}$. 
This contradicts the oddness of $S$. 
\end{proof}

\section{Appendix: Connections with Tropical Geometry}
In the following section no new results are obtained, so the style will be rather informal.

It is more natural to define tropical colorings on  $\mathds{R}\mathds{P} ^2$ --- the real projective plane. It is well-known that a point of $\mathds{R}\mathds{P} ^2$ is defined by its homogenious coordinates --- a triple of real numbers $[x:y:z]$ with not all $x, y, z$ equal to $0$. For any nonzero $\lambda$ triples $[x:y:z]$ and $[ \lambda x:\lambda y:\lambda z]$ define the same point.
One can define a \textbf{momentum map} from the projective plane  to a triangle $T$ in the plane with vertices (1,0),(0,1) and (0,0):
$$m:\mathds{R}\mathds{P} ^2 \longrightarrow T$$
defined by the formula
$$m([x:y:z])=\frac{2^{-\nu_2(x)} \Big( 1,0\Big)  + 2^{-\nu_2(y)} \Big( 0,1\Big)+ 2^{-\nu_2(z)} \Big( 0,0\Big)}{2^{-\nu_2(x)}+2^{-\nu_2(y)}+2^{-\nu_2(z)}}.$$

\begin{figure}[hbtp]
\centering
\includegraphics[width=2.5in]{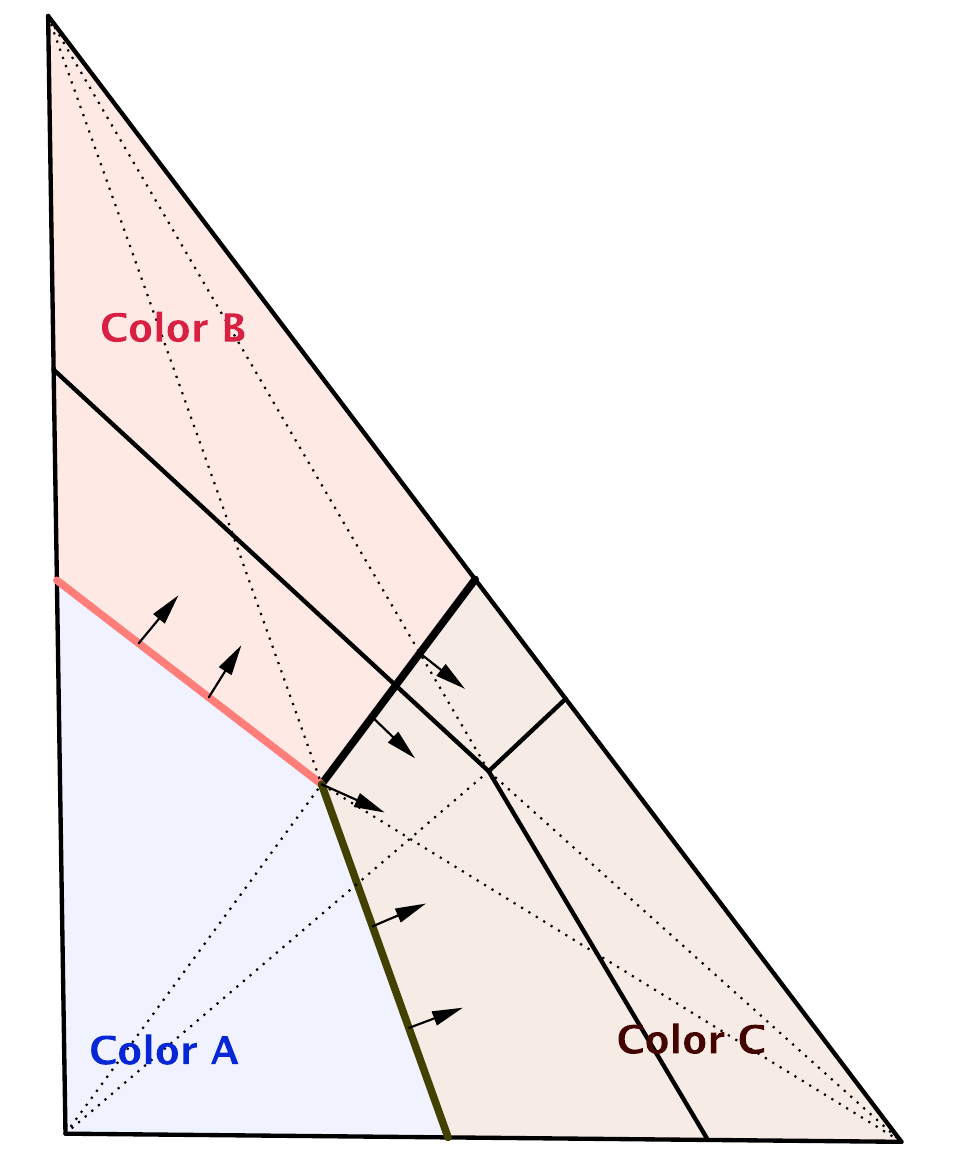}
\caption{Tropical coloring of the plane and the image of the line}
\label{Tropical coloring}
\end{figure}

One can check that the image of any line in $\mathds{R}\mathds{P} ^2$ under the momentum map is a union of three segments sharing a common end. For each segment the remaining end is lying on a side of triangle $T$ and the whole segment is contained within a line, passing through the vertex of $T$ opposite to the side.

The image of the line $x+y+z=0$ cuts T into three pieces, whose points we color in three colors A, B and C. Now we can color each point in $\mathds{R}\mathds{P} ^2$ in the color of its image under the momentum map. This coloring is the same as  coloring $\pi$ constructed in the beginning of the paper if considered on the affine chart of $\mathds{R}\mathds{P} ^2$  with $z=1$.

Property \textbf{P1} is obvious now --- one can see that the image of any other line under the momentum map can intersect only two parts in which the image of the line $x+y+z=0$ cuts triangle T. 

In article \cite{HS} by A. Hales and E. Straus  colorings of $\mathds{R}\mathds{P} ^2$  are studied in more detail. One of their results is the following theorem:

\begin{theorem} [A. Hales, E. Straus, 1982 ] \label{HS}  
Let $C$ be a set of algebraic curves in $\mathds{R}^2$ having the same Newton polygon $P$ with $n$ integer points inside ($C$ is a $n$--dimensional linear system of algebraic curves). Then there exists a coloring of $\mathds{R}^2$ in $n + 1$ colors such that no curve in $C$ contains all $n + 1$ colors and no color is confined within a curve in $C$.\footnote{Actually, the result obtained in \cite{HS}  is  stronger: it holds for colorings of a projective plane over any field which has a nontrivial non--Archimedean valuation, and for arbitrary  $n$--dimensional linear systems of algebraic curves without based points.}
\end{theorem} 

As a specific instance of this theorem one gets a coloring of the plane in six colors, such that any conic contains at most five colors. 
\begin{figure}[hbtp]
\centering
\includegraphics[width=4.5in]{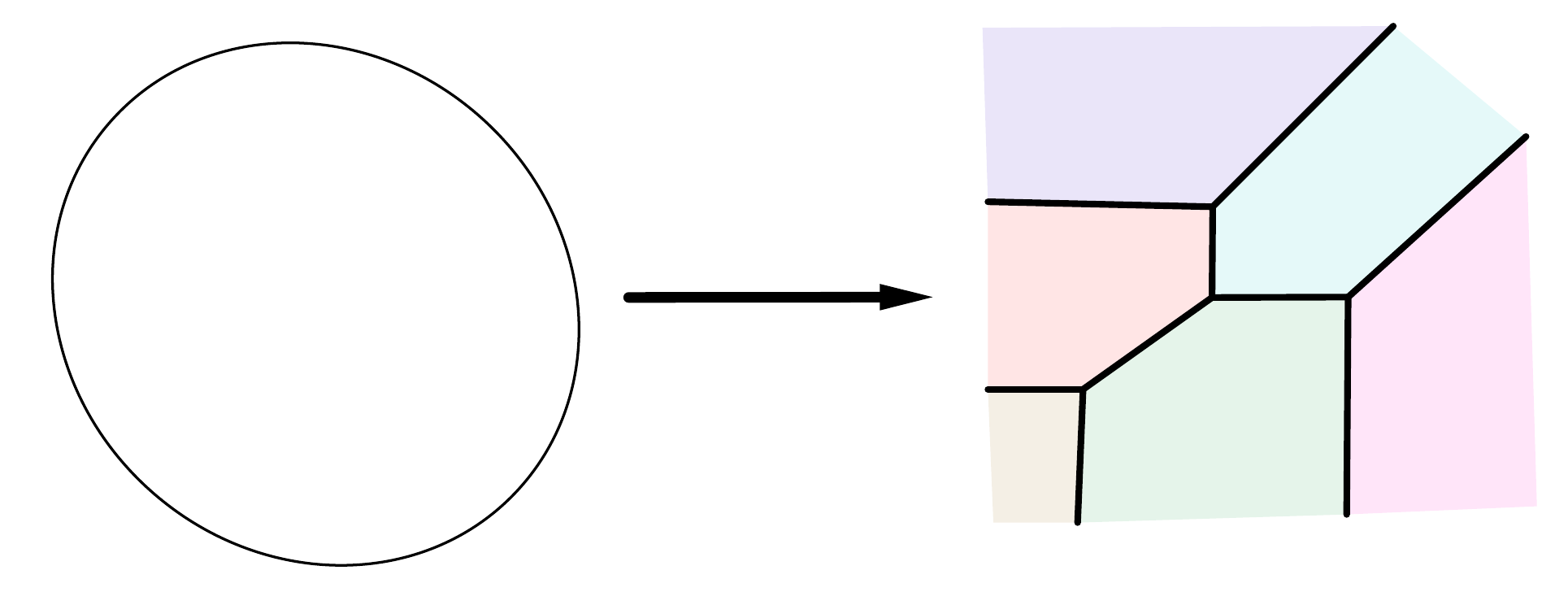}
\caption{The image of a conic under the  momentum map and the corresponding coloring}

\end{figure}

In \cite{HS} such colorings are constructed algebraically; here we will give a tropical explanation of this construction.  By $\mathds{R}^* $ we mean $\mathds{R} \setminus \{ 0 \}$. Let us consider algebraic torus $\mathds{R}^* \times \mathds{R}^*$ and "a momentum map":
$$m:\mathds{R}^* \times \mathds{R}^* \longrightarrow \mathds{R} ^2, $$
where
$$m \Big(  (x,y) \Big)=(\nu_2(x),\nu_2(y)) \in \mathds{R} ^2.$$

There exists a curve with Newton polygon $P$, whose image devides $\mathds{R} ^2$ into $n+1$ regions. We assign to them different colors.   It can be proved that the image of any other curve with Newton Polygon $P$ intersects at most $n$ regions. Now, we can color each point in $\mathds{R}^* \times \mathds{R}^* \in  \mathds{R} ^2$ in  the color of the region of $\mathds{R}^2$ containing its image.  This coloring is the same as that constructed by Hales and Straus.

This illustration shows that  colorings constructed in \cite{HS} are natural from the perspective, suggested by tropical geometry.  Unfortunatelly, it does not lead to an easier way of proving Theorem \ref{HS}, because of both combinatorial difficulties in analyzing the way in which two tropical curves intersect and algebraic difficulties in extending  the colorings from $\mathds{R}^* \times \mathds{R}^* $ to the whole $ \mathds{R} ^2$.

\end{document}